\documentclass[12pt]{amsart}

\usepackage{amsmath,amssymb,amsthm}

\newtheorem{thm}{Theorem}[section]

\newtheorem{prop}[thm]{Proposition}
\newtheorem{lem}[thm]{Lemma}
\newtheorem{cor}[thm]{Corollary}

\theoremstyle{definition}
\newtheorem{exam}[thm]{Example}
\newtheorem{defn}[thm]{Definition}
\newtheorem{rem}[thm]{Remark}

\begin{document}
\title[Fragmented Hofer's geometry]{Fragmented Hofer's geometry}
\author[Morimichi Kawasaki]{Morimichi Kawasaki}

\address[Morimichi Kawasaki]{Graduate School of Mathematical Sciences, the University of Tokyo, 3-8-1 Komaba, Meguro-ku, Tokyo 153-0041, Japan and Center for Geometry and Physics, Institute for Basic Science (IBS), Pohang 37673, Republic of Korea}

\email{kawasaki@ibs.re.kr}

\begin{abstract}
Hofer's norm (metric) is an important and interesting topic in symplectic geometry.
In the present paper, we define fragmented Hofer's norms which are Hofer's norms controlled by fragmentation norms and give some observations on fragmented Hofer's norms.
\end{abstract}
\maketitle
\section{Introduction}
\large

Let $(M,\omega)$ be a symplectic manifold.

For a Hamiltonian function $F\colon M\to\mathbb{R}$ with compact support, we define the \textit{Hamiltonian vector field} $X_F$ associated with $F$ by
\[\omega(X_F,V)=-dF(V)\text{ for any }V \in \mathcal{X}(M),\]
where $\mathcal{X}(M)$ is the set of smooth vector fields on $M$.

Let $a,b$ be real numbers with $a<b$.
For a (time-dependent) Hamiltonian function $F\colon  M\times[a,b]\to\mathbb{R}$ with compact support and $t \in [a,b]$, let $F_t\colon M\to\mathbb{R}$ denote the time-independent function defined by $F_t(x)=F(x,t)$. 
Let $X_F^t$ denote the Hamiltonian vector field associated with $F_t$ and $\{\phi_F^t\}_{t\in[a,b]}$ denote the isotopy generated by $X_F^t$ with $\phi_F^a=\mathrm{id}$. 
The time-b map $\phi_F^b$ of $\{\phi_F^t\}$ is called \textit{the Hamiltonian diffeomorphism generated by the Hamiltonian function $F$} and denoted by $\phi_F$.
For a symplectic manifold $(M,{\omega})$, let $\mathrm{Ham}(M,\omega)$ denote the group of Hamiltonian diffeomorphisms of $(M,{\omega})$.

We define the Hofer's length of a Hamiltonian function $F\colon M\times[a,b]\to\mathbb{R}$ as follows:
\[||F||=\int^b_a(\max_M F_t-\min_M F_t)dt.\]
We review the definition of the original Hofer's norm $||\cdot||$(\cite{Ho}).

\begin{defn}\label{definition of Hofer}
Let $(M,\omega)$ be a symplectic manifold.
For a Hamiltonian diffeomorphism $\phi$, we define \textit{Hofer's norm} of $\phi$ by
\[||\phi||=\inf\{ ||F||\},\]
where the infimum is taken over Hamiltonian functions $F\colon M\times[0,1]\to\mathbb{R}$ such that $\phi=\phi_F$. 
\end{defn}

The original Hofer's norm $||\cdot||$ is known to be non-degenerate \textit{i.e.} $||\phi||$ is positive for any Hamiltonian diffeomorphism $\phi$ which is not the identity map $1$(\cite{LM}).

In \cite{Ka}, the author studied the commutator length controlled by a fragmentation norm (see also \cite{Kim}).
In the present paper, we study the Hofer's norm controlled by a fragmentation norm which we call a fragmented Hofer's norm.
For a time-dependent Hamiltonian function $H\colon M\times[a,b]\to\mathbb{R}$, we define the support $\mathrm{Supp}(H)$ of $H$ by $\mathrm{Supp}(H)=\bigcup_{[a,b]}\mathrm{Supp}(H_t)$.
\begin{defn}\label{definition of fragmented Hofer}
Let $(M,\omega)$ be a symplectic manifold and $U$ be a non-empty open subset of $M$.
We define \textit{Hofer's norm of $\phi$ fragmented by} $U$ by
\[||\phi||_U=\inf\{ ||F^1||+\cdots +||F^k|| \},\]
where the infimum is taken over Hamiltonian functions $F^1,\ldots ,F^k\colon  M\times[0,1]\to\mathbb{R}$ and Hamiltonian diffeomorphisms $h_1,\ldots,h_k$ such that $\mathrm{Supp}(F^i)\subset U$ for any $i$ and $\phi=h_1\phi_{F^1}h_1^{-1}\cdots h_k \phi_{F^k}h_k^{-1}$.
\end{defn}

Banyaga's fragmentation lemma (\cite{Ba}) states that the above decomposition exists and thus $||\cdot||_U$ is well-defined \textit{i.e.} $||\phi||_U<\infty$ for any Hamiltonian diffeomorphism $\phi$.
By the definition, we see that
\begin{itemize}
\item the original Hofer's norm is equal to Hofer's norm fragmented by the whole symplectic manifold $M$ \textit{i.e.} $||\phi||=||\phi||_M$ holds for any Hamiltonian diffeomorphism $\phi$,
\item For open subsets $U,V$ of $M$ with $U\subset V$, $||\phi||_V\leq||\phi||_U$ holds for any Hamiltonian diffeomorphism $\phi$.
\end{itemize}
In particular, $||\phi||\leq ||\phi||_U$ holds for any Hamiltonian diffeomorphism $\phi$ and any open subset $U$ of $M$.

We also easily verify that the following proposition holds.

\begin{prop}\label{cin}
Let $(M,\omega)$ be a symplectic manifold and $U$ an open subset of $M$.
Then $||\cdot||_U$ is a conjugation-invariant norm in the sense of Burago, Ivanov and Polterovich(\cite{BIP}) \textit{i.e.} $||\cdot||_U$ satisfies the following conditions.
\begin{itemize}
\item[(1)] $||1||_U=0$,
\item[(2)] $||f||_U=||f^{-1}||_U$ for any Hamiltonian diffeomorphism $f$,
\item[(3)] $||fg||_U\leq ||f||_U+||g||_U$ for any Hamiltonian diffeomorphisms $f,g$,
\item[(4)] $||f||_U=||gfg^{-1}||_U$ for any Hamiltonian diffeomorphisms $f,g$,
\item[(5)] $||f||_U>0$ for any Hamiltonian diffeomorphism $f$ with $f\neq 1$.
\end{itemize}
\end{prop}

Conjugation-invariant norms $||\cdot||_1$ and $||\cdot||_2$ on $G$ are \textit{equivalent} if there exists positive numbers $a$ and $b$ such that $\frac{1}{a}||\phi||_1-b\leq ||\phi||_2\leq a||\phi||_1+b$ for any element $\phi$ of $G$.

We state that any fragmented Hofer's norm is equivalent to the original Hofer's norm on a compact symplectic manifold.
\begin{thm}\label{equiv}
Let $(\hat{M},\omega)$ be a compact symplectic manifold which can have a smooth boundary and $M$ the interior of $\hat{M}$.
For any open subset $U$ of $M$, there is a positive number $C_U$ such that $||\phi||_U\leq C_U||\phi||$ for any Hamiltonian diffeomorphism $\phi$ on $(M,\omega)$.
\end{thm}
We prove Theorem \ref{equiv} in Section \ref{equiv section}.

However, some fragmented Hofer's norms are not equivalent to the original Hofer's norm on $\mathrm{Ham}(\mathbb{R}^{2n},\omega_0)$ where $\omega_0=dx_1\wedge dy_1+\cdots +dx_n\wedge dy_n$ is the standard symplectic form on the Euclidean space $\mathbb{R}^{2n}$ with coordinates $(x,y)=(x_1,\ldots,x_n,y_1,\ldots,y_n)$.

To prove it, we introduce the Calabi homomorphism.
A symplectic manifold $(M,\omega)$ is called \textit{exact} if the symplectic form $\omega$ is exact.
\begin{defn}[{\cite{Ba}}]\label{Calabi homomorphism}
Let $(M,\omega)$ be a $2n$-dimensional exact symplectic manifold.
\textit{The Calabi homomorphism} $\mathrm{Cal}{\colon}\mathrm{Ham}(M,\omega)\to\mathbb{R}$ is defined by
\[\mathrm{Cal}(h)=\int_0^1\int_MH_t{\omega}^ndt\text{ for a Hamiltonian diffeomorphism }h,\]
where $H\colon M\times[0,1]\to\mathbb{R}$ is a Hamiltonian function which generates $h$.
$\mathrm{Cal}(h)$ does not depend on the choice of generating Hamiltonian function $H$ (\cite{Ba} and \cite{Hu}). Thus the functional $\mathrm{Cal}$ is a well-defined homomorphism.
\end{defn}
By using the Calabi homomorphism, we obtain the following lower bound of fragmented Hofer's norms.
\begin{prop}\label{Calabi}
Let $U$ be an open subset of a $2n$-dimensional exact symplectic manifold $(M,\omega)$.
If $\mathrm{Vol}(U,\omega)<\infty$, then for any Hamiltonian diffeomorphism $\phi$,
\[||\phi||_U\geq\mathrm{Vol}(U,\omega)^{-1}|\mathrm{Cal}(\phi)|,\]
where $\mathrm{Vol}(U,\omega)$ is the volume $\int_U 1 \omega^n$ of $U$.
\end{prop}

It is known $||\phi_H||\geq\mathrm{Vol}(M,\omega)^{-1}|\mathrm{Cal}(\phi_H)|$ on the original Hofer's norm.
We prove Proposition \ref{Calabi} in Section \ref{equiv section}.

Sikorav \cite{S} (see also subsection 5.6 of \cite{HZ}) proved that 
the original Hofer's norm $||\cdot||$ on the Euclidean space is stably bounded \textit{i.e.} $\lim_{k\to\infty}\frac{||\phi^k||}{k}=0$ for any Hamiltonian diffeomorphism $\phi$ on $\mathbb{R}^{2n}$.
Thus Proposition \ref{Calabi} implies the following corollary.

\begin{cor}
Let $U$ be an open subset of $\mathbb{R}^{2n}$ with $\mathrm{Vol}(U,\omega_0)<\infty$.
Then the fragmented Hofer's norm $||\cdot||_U\colon\mathrm{Ham}(\mathbb{R}^{2n},\omega_0)\to\mathbb{R}$ is not equivalent to the original Hofer's norm $||\cdot||$.
\end{cor}

The author does not know whether the fragmented Hofer's norm $||\cdot||_U$ is equivalent to the original Hofer's norm $||\cdot||$ on $\mathrm{Ham}(\mathbb{R}^{2n},\omega_0)$ when $U=\{|x_1|^2+|y_1^2|<1\}$ and $n\geq2$.

\begin{defn}
Let $(M,\omega)$ be a symplectic manifold and $\mathcal{U}=\{U_i\}_{i=1,2\ldots}$ a sequence of open subsets in $M$.
$\mathcal{U}$ has \textit{uniformly bounded fragmented Hofer's geometry (UBFH)} if there are a non-trivial Hamiltonian diffeomorphism $\phi$ and a positive number $C$ such that $||\phi||_{U_i}<C$ for any $i$.
\end{defn}

If $(M,\omega)$ is an exact symplectic manifold or a closed symplectic manifold, UBFH naturally induces a stronger property.
\begin{prop}\label{rigidity}
Let $(M,\omega)$ be a symplectic manifold and $\mathcal{U}=\{U_i\}_{i=1,2\ldots}$ a sequence of open subsets in $M$ with  $\lim_{i\to\infty}\mathrm{Vol}(U_i,\omega)=0$ and UBFH.
\begin{itemize}
\item Let $(M,\omega)$ be an exact one.
Then, for any Hamiltonian diffeomorphism $\phi$, there exists a positive number $C_\phi$ such that $||\phi||_{U_i}<C_{\phi}$ for any $i$ if and only if $\mathrm{Cal}(\phi)=0$.
\item Let $(M,\omega)$ be a closed one.
Then, for any Hamiltonian diffeomorphism $\phi$, there exists a positive number $C_\phi$ such that $||\phi||_{U_i}<C_{\phi}$ for any $i$.
\end{itemize}
\end{prop}
\begin{proof}
Define a subset $N_\mathcal{U}$ of $\mathrm{Ham}(M,\omega)$ by
\[N_\mathcal{U}=\{\phi\in\mathrm{Ham}(M,\omega); \exists C_\phi>0\text{ such that }\forall i, ||\phi||_{U_i}<C_\phi\}.\]
Since $N_\mathcal{U}$ is a conjugation-invariant norm, $N_\mathcal{U}$ is a normal subgroup of the group $\mathrm{Ham}(M,\omega)$.
Since $\mathcal{U}$ has UBFH, $N_\mathcal{U}$ is non-trivial.

 Let $(M,\omega)$ be an exact symplectic manifold. Then, by Proposition \ref{Calabi} and $\lim_{i\to\infty}\mathrm{Vol}(U_i,\omega)=0$, $N_\mathcal{U}\subset\mathrm{Ker}(\mathrm{Cal})$.
Thus, since $\mathrm{Ker}(\mathrm{Cal})$ is a simple group (\cite{Ba}), $N_\mathcal{U}=\mathrm{Ker}(\mathrm{Cal})$.

Let $(M,\omega)$ be a closed symplectic manifold.
Then, since $\mathrm{Ham}(M,\omega)$ is a simple group (\cite{Ba}), $N_\mathcal{U}=\mathrm{Ham}(M,\omega)$.
\end{proof}

Any symplectic manifold $(M,\omega)$ admits a sequence $\mathcal{U}=\{U_i\}_{i=1,2,\ldots}$ of open subsets in $M$ with $\lim_{i\to\infty}\mathrm{Vol}(U_i,\omega)=0$ and UBFH.
To construct such a sequence, we introduce some notions.

For a positive number $r$, let $Q_r$ be a cube defined by
\[Q_r =\{(x,y)\in\mathbb{R}^{2n} ; |x_1|^2+\cdots+|x_n|^2+|y_1|^2+\cdots+|y_n|^2< r^2\}.\]

For a positive integer $l$ and a subset $X$ of $\mathbb{R}^{2n}$, let $\partial_lX$ denote the ${1}/{l}$-neighborhood of the boundary $\partial X$ of $X$ in $\mathbb{R}^{2n}$.
Then we obtain the following theorem.
\begin{thm}\label{uniformly bounded}
Let $(M,\omega)$ be a symplectic manifold and  $\iota\colon Q_{5r}\to M$ a symplectic embedding.
Let $W_l$ denote the open subset $\iota(\partial_lQ_r)$ of $M$ and $||\cdot||_l$ denote the norm $||\cdot||_{W_l}$.
There exists a positive constant $C$ such that  $||[\phi_F,\phi_G]||_l<C$ for any Hamiltonian functions $F,G\colon M\times[0,1]\to\mathbb{R}$ whose support is in $\iota(Q_r)$ and any positive integer $l$.
In particular, $\{W_l\}_{l=1,2,\ldots}$ has UBFH.
\end{thm}
We prove Theorem \ref{uniformly bounded} in Section \ref{UBFH section}.
The author does not know whether the sequence $\{\iota(Q_{r/l})\}_{l=1,2,\ldots}$ has UBFH or not.

Many researchers have studied functionals $c\colon C_c^\infty(M\times [0,1])\to\mathbb{R}$ which are called \textit{spectral invariants} or \textit{action selectors} (\cite{V}, \cite{HZ}, \cite{Sc}, \cite{FS}, \cite{Oh05}, \cite{EnP03} and \cite{FOOO}).
We obtain a lower bound of fragmented Hofer's norm by spectral invariants.

\begin{prop}\label{independence}
Let $(M,\omega)$ be an exact symplectic manifold.
Assume that a functional $c\colon C_c^\infty(M\times [0,1])\to\mathbb{R}$ satisfies the following conditions.
\begin{description}
\item[(1)\textit{invariance}] Assume that Hamiltonian functions $F$ and $G\colon  M\times[0,1]\to\mathbb{R}$ satisfy $\phi^1_F=\phi^1_G$. Then $c(a,F)=c(a,G)$.

\item[(2)\textit{triangle inequality}]$c(F\sharp G) \leq c(F)+c(G)$ for any Hamiltonian functions $F,G\colon M\times[0,1]\to\mathbb{R}$.
Here $(F{\sharp}G)$ is the Hamiltonian function defined by $(F{\sharp}G)(x,t)=F(x,t)+G((\phi_F^t)^{-1}(x),t)$ whose Hamiltonian isotopy is $\{\phi_F^t\phi_G^t\}$.
\item[(3)\textit{stability}]$-\int^1_0\max_M(F_t-G_t)dt\leq c(F)-c(G)\leq-\int^1_0\min_M (F_t-G_t)dt$ for any Hamiltonian functions $F,G\colon M\times[0,1]\to\mathbb{R}$.
\item[(4)\textit{conjugation invariance}]$c(H\circ \phi)=c(H)$ for any Hamiltonian function $H\colon M\times[0,1]\to\mathbb{R}$ and any Hamiltonian diffeomorphism $\phi$.
\item[(5)]$c(0)=0$.
\end{description}
Then, for any Hamiltonian function $H\colon M\times[0,1]\to\mathbb{R}$,
\[c(H)+\mathrm{Vol}(U,\omega)^{-1}\mathrm{Cal}(\phi_H)\leq ||\phi_H||_U.\]
\end{prop}
It is known that 
\[c(H)+\mathrm{Vol}(M,\omega)^{-1}\mathrm{Cal}(\phi_H)\leq ||\phi_H||,\] 
on the original Hofer's norm.
We prove Proposition \ref{independence} in Section \ref{spec section}.
\begin{exam}
Frauenfelder and Schlenk (\cite{FS}) proved that there exists a spectral invariant $c\colon C_c^\infty(M\times [0,1])\to\mathbb{R}$ which satisfies the conditions of Proposition \ref{independence} if $(M,\omega)$ is an exact compact convex symplectic manifold. 

The most general construction of spectral invariants is Oh's one
(\cite{Oh05}) and its bulk-deformation (\cite{FOOO}).
However, Oh's spectral invariant does not satisfy the above condition (1) in general.
For conditions to satisfy the above condition (1), refer \cite{Sc}, \cite{EnP03}, \cite{M} and \cite{Se}.
\end{exam}

\begin{rem}\label{another spec}
There are other conventions of spectral invariants.
For instance, under the convention of \cite{EnP03}, spectral invariants satisfy the following condition instead of the above condition (4).
\begin{description}
\item[$(4)^\prime$]$\int^1_0\min_M (F_t-G_t)dt\leq c(F)-c(G)\leq\int^1_0\max_M(F_t-G_t)dt$ for any Hamiltonian functions $F,G\colon M\times[0,1]\to\mathbb{R}$.
\end{description}
If a functional $c\colon C_c^\infty(M\times [0,1])\to\mathbb{R}$ satisfies the conditions (1),(2),(3),$(4)^\prime$ and (5), the following inequality holds instead of Proposition \ref{independence}.
\[c(H)-\mathrm{Vol}(U,\omega)^{-1}\mathrm{Cal}(\phi_H)\leq ||\phi_H||_U.\]
For the proof of this inequality, see Remark \ref{another c^U}.
\end{rem}

\begin{rem}\label{asymptotic}
In many cases, asymptotic spectral invariants satisfy the conditions of Proposition \ref{independence}.
For instance, Monzner-Vichery-Zapolsky's spectral invariant (\cite{MVZ}) does not satisfy the condition (4), but the asymptotization of their spectral invariant satisfies the condition (4).
\end{rem}

\subsection*{Acknowledgment}
The author would like to thank Kei Irie and Professor Leonid Polterovich and his advisor Professor Takashi Tsuboi.
The communications with them influenced the present paper very much.
He also thanks Tomohiko Ishida, Yaron Ostrover and Daniel Rosen for their comments.

He is supported by IBS-R003-D1, the Grant-in-Aid for Scientific Research (KAKENHI No. 25-6631) and the Grant-in-Aid for JSPS fellows. This work was 
supported by the Program for Leading Graduate Schools, MEXT, Japan.

\section{Proof of Theorem \ref{equiv} and Proposition \ref{Calabi}}\label{equiv section}
To prove Theorem \ref{equiv}, we use the following lemma.
\begin{lem}\label{equiv lem}
Let $(\hat{M},\omega)$ be a compact symplectic manifold which can have a smooth boundary and $M$ the interior of $\hat{M}$.
For any open subset $U$ of $M$, there is a positive number $C_U$ such that $||\phi_H||_U\leq C_U||H||$ for any $C^1$-small Hamiltonian function $H\colon M\times[0,1]\to\mathbb{R}$.
\end{lem}

For a subset $X$ of $\hat{M}$, the topological closure of $X$ is denoted by $\bar{X}$.

\begin{proof}
Since $\hat{M}$ is compact, we can take finite open coverings $\mathcal{V}=\{V_i\}_{i=1,\ldots,l}$ and $\hat{\mathcal{V}}=\{\hat{V}_i\}_{i=1,\ldots,l}$ of $\hat{M}$ such that
\begin{itemize}
\item  $\bar{V_i}\subset\hat{V}_i$,
\item  for any compact subset $K$ of $M$ and $i$, there exists a Hamiltonian diffeomorphism $\phi_{i,K}$ with compact support in $M$ such that $\phi_{i,K}(\overline{\hat{V}_i}\cap K)\subset U$.
\end{itemize}

For $j=0,\ldots,l$, let $W_j$ and $\hat{W}_j$ denote $\bigcup_{i=1}^jV_i$ and $\bigcup_{i=1}^j\hat{V}_i$, respectively. Now, we define $W_0$ and $\hat{W}_0$ as the empty set $\emptyset$.

Since $\overline{(\hat{W}_{j+1}\setminus \hat{V}_{j+1})}\cap\overline{(V_{j+1}\setminus \hat{W}_j)}=\emptyset$, we can take smooth functions $\rho_j\colon M\to[0,1]$ ($j=0,\ldots,l-1$) such that
\begin{itemize}
\item $\rho_j=1$ on some open neighborhood of $\overline{(\hat{W}_{j+1}\setminus \hat{V}_{j+1})}$,
\item $\rho_j=0$ on some open neighborhood of $\overline{(V_{j+1}\setminus \hat{W}_j)}$.
\end{itemize}
Let $\chi_l$ be a constant function $1$ on $M$ and define fuctions $\chi_j\colon M\to[0,1]$ (j=l-1,\ldots,0) inductively by $\chi_j=\rho_j\cdot\chi_{j+1}$.
Then we define Hamiltonian functions $K_j$ (j=0,\ldots,l) and $L_j$ (j=0,\ldots,l) by
\begin{itemize}
\item $K^j(x,t)=\chi_j(x)\cdot H(x,t)$,
\item $L^{j+1}(x,t)=-K^j(\phi_{K^j}^t(x),t)+K^{j+1}(\phi_{K^j}^t(x),t)$.
\end{itemize}
Note that $L^{j+1}$ generates the Hamiltonian diffeomorphism $\phi_{K^j}^{-1}\phi_{K^{j+1}}$ and thus $\phi_{K^{j+1}}=\phi_{K^j}\phi_{L^{j+1}}$.
Since $K^l(x,t)=H(x,t)$ and $K^0(x,t)=0$, $\phi_H=\phi_{K^l}=\phi_{K^{l-1}}\phi_{L^{l}}=\ldots=\phi_{K^0}\phi_{L^1}\cdots\phi_{L^l}=\phi_{L^1}\cdots\phi_{L^l}$.

Now, we claim $\mathrm{Supp}(L^{j})\subset \overline{\hat{V}_j}$.
Fix a point $x$ in $M$ with $x\notin \overline{\hat{V}_j}$.
Note that $M\setminus\overline{\hat{V}_j}\subset(\hat{W}_j\setminus \hat{V}_j)\cup(M\setminus(V_j\setminus \hat{W}_{j-1}))$ since $M\setminus \hat{W}_j\subset M\setminus (V_j\setminus \hat{W}_{j-1})$.
\begin{itemize}
\item
Assume $x\in \hat{W}_j\setminus \hat{V}_j\subset \hat{W}_{j-1}$.
Since $H$ is a $C^1$-small Hamiltonian function, $K^{j-1}$ is also $C^1$-small.
Thus $\phi_{K^{j-1}}^t(\hat{W}_j\setminus \hat{V}_j)\subset\mathrm{Supp}(1-\rho_{j-1})$ for any $t$ and therefore $\chi_{j-1}(\phi_{K^{j-1}}^t(x))=\chi_{j}(\phi_{K^{j-1}}^t(x))$ for any $t$.
Hence $L^j(x,t)=0$ for any $t\in[0,1]$.
\item
Assume $x\notin V_j\setminus \hat{W}_{j-1}$.
Since $K^{j-1}$ is also $C^1$-small, $\phi_{K^{j-1}}^t(V_j\setminus \hat{W}_{j-1})\subset\mathrm{Supp}(\rho_{j-1})$.
Thus $\chi_{j-1}(\phi_{K^{j-1}}^t(x))=\chi_{j}(\phi_{K^{j-1}}^t(x))=0$ and therefore $L^j(x,t)=0$ for any $t\in[0,1]$.

\end{itemize}
Hence $L^j(x,t)=0$ for any $x\notin \overline{\hat{V}_j}$ and any $t\in[0,1]$ and we complete the proof of $\mathrm{Supp}(L^{j})\subset \overline{\hat{V}_j}$.

By $\mathrm{Supp}(L^{j})\subset \overline{\hat{V}_j}$ and the first condition of $\hat{\mathcal{V}}$, there exists a Hamiltonian diffeomorphism $h_j$ such that $h_j(\mathrm{Supp}(L^j))\subset U$.
Therefore $||\phi_{L^j}||_U=||h_j\phi_{L^j}h_j^{-1}||\leq ||L^j||\leq ||H||$ and
\[
||\phi_H||_U\leq ||\phi_{L^1}||_U+\cdots+||\phi_{L^l}||_U \leq l\cdot ||H||.
\]
\end{proof}

\begin{proof}[Proof of Theorem \ref{equiv}]
Let $H\colon M\times[0,1]\to\mathbb{R}$ be a Hamiltonian function.
For positive integers $N$ and $n$ with $n\leq N$, we define a Hamiltonian function $H^{n,N}\colon M\times[0,1]\to\mathbb{R}$ by
\[H^{n,N}(x,t)=\frac{1}{N}H(x,\frac{n-1+t}{N}).\]
Then $\phi_{H^{n,N}}=\phi_{H|_{M\times [(n-1)/N,n/N]}}$ holds.
If $N$ is sufficiently large, then $H^{n,N}$ is sufficiently $C^1$-small to satisfy the assumption of Lemma \ref{equiv lem} for any $n=1,\ldots,N$.
Thus Lemma \ref{equiv lem} implies
\begin{align*}
&||\phi_H||_U\leq||\phi_{H^{1,N}}||_U+\cdots+ ||\phi_{H^{n,N}}||_U\\
& \leq C_U||H^{1,N}||+\cdots+ C_U||H^{n,N}||\leq N\cdot \frac{C_U}{N}||H||=C_U||H||.
\end{align*}
By taking the infimum over a Hamiltonian function $H$ which generates $\phi$, we complete the proof of Theorem \ref{equiv}.
\end{proof}

\begin{proof}[Proof of Proposition \ref{Calabi}]
Let $F^1,\ldots,F^k$ and $h_1,\ldots,h_k$ be Hamiltonian functions and Hamiltonian diffeomorphisms such that $\mathrm{Supp}(F^i)\subset U$ for any $i$ and $\phi=h_1\phi_{F^1}h_1^{-1}\cdots h_k\phi_{F^k}h_k^{-1}$, respectively.
By the definition of the Calabi homomorphism, $|\mathrm{Cal}(\phi_{F^i})|\leq \mathrm{Vol}(U,\omega)||F^i||$ for any $i$.
Since the Calabi homomorphism is a homomorphism,
\begin{align*}
&|\mathrm{Cal}(\phi)|\\
& =|\mathrm{Cal}(\phi_{F^1})|+\cdots +|\mathrm{Cal}(\phi_{F^k})|\\
&\leq \mathrm{Vol}(U,\omega)||F^1||+\cdots + \mathrm{Vol}(U,\omega)||F^k||\\
      &= \mathrm{Vol}(U,\omega)(||F^1||+\cdots+||F^k||).
\end{align*}
By taking the infimum, we complete the proof.
\end{proof}

\section{Uniformly bounded fragmented Hofer's geometry}\label{UBFH section}

To prove Theorem \ref{uniformly bounded}, we use the beautiful argument by Eliashberg and Polterovich (\cite{ElP93}, see also Section 2 of \cite{P}).

\begin{proof}[Proof of Theorem \ref{uniformly bounded}]
For any $0\leq t\leq 1$, define a subset $Q_{3r}^t$ of $Q_{3r}$ by
\[Q_r^t=\{(x,y)\in\mathbb{R}^{2n}; (x_1+t,x_2,\ldots,x_n,y)\in Q_r\}.\]
For a sufficiently large integer $l$, we can take a Hamiltonian function $H\colon Q_{5r}\times [0,1]\to\mathbb{R}$ such that
\begin{itemize}
\item $H(x,y,t)=3ry_1$ if $(x,y)\in \partial_{3l} Q_r^{3rt}$,
\item $\mathrm{Supp}H_t\subset \partial_{2l}Q_r^{3rt}$,
\item $-4r^2\leq H(x,y,t)\leq 4r^2$.
\end{itemize}
Since $\mathrm{Supp}H_t\subset \partial_{2l}Q_r^{3rt}$ for any $t$, $\bigcup_{t\in [\frac{2i}{2l},\frac{2i+2}{2l}]}\mathrm{Supp}H_t\subset \partial_lQ_r^{\frac{3r(2i+1)}{2l}}$ holds for any $i=0,\cdots,l-1$.
Let $H^i$ be the restriction of $H\colon M\times[0,1]\to\mathbb{R}$ to $M\times[\frac{2i}{2l},\frac{2i+2}{2l}]$ for $i=0,\cdots,l-1$.
Then $\phi_H=\phi_{H^{l-1}}\cdots \phi_{H^{0}}$.
Since $-4r^2\leq H(x,y,t)\leq 4r^2$, $||H^i||\leq\frac{8r^2}{l}$.
Since $\partial_lQ_r^t$ and $\partial_lQ_r$ are conjugate by a Hamiltonian diffeomorphism on $Q_{5r}$,
\[||\phi_H||_{\partial_lQ_r}\leq||H^0||+\cdots +||H^{l-1}||\leq 8r^2.\]
Since $H(x,y,t)=3ry_1$ for any $(x,y)\in \partial_{3l}Q_r^{3rt}$, $\phi_H^t(\partial Q_r)=\partial Q_r^{3rt}$.
In particular, $\phi_H(\partial Q_r)\cap \partial Q_r=\emptyset$.
Thus $\phi_H(Q_r)\cap Q_r=\emptyset$.
Now, we regard $H$ as a Hamiltonian function on $M$ through $\iota$.
Then, for any two Hamiltonian functions $F,G\colon \iota(Q_r)\times[0,1]\to\mathbb{R}$ on $\iota(Q_r)$, $[\phi_F,\phi_G]=[\phi_F,[\phi_G,\phi_{H }]]$ holds.
Since $||\cdot||_l$ is a conjugation-invariant norm,
\begin{align*}
&||[\phi_F,[\phi_G,\phi_{H }]||_l\\
& \leq ||\phi_F[\phi_G,\phi_{H }](\phi_F)^{-1}||_l+||[\phi_G,\phi_{H }]^{-1}||_l\\
&=2||[\phi_G,\phi_{H }]||_l\\
      & \leq 2(||\phi_G\phi_{H }\phi_G^{-1}||_l+||(\phi_{H })^{-1}||_l)\\
      &=4||\phi_{H }||_l\leq4||\phi_H||_{\partial_lQ_r}\leq 32r^2.
\end{align*}
Thus $||[\phi_F,\phi_G]||_l=||[\phi_F,[\phi_G,\phi_H]||_l\leq 32r^2$.
\end{proof}

\section{Lower bound by spectral invariants}\label{Estimation by spectral invariants}\label{spec section}

Let $(M,\omega)$ be an exact symplectic manifold, $c\colon C_c^\infty(M\times[0,1])\to\mathbb{R}$ a functional satisfying the conditions of Proposition \ref{independence} and $U$ an open subset of $M$.
We define functionals $c^U\colon C_c^\infty(U\times [0,1])\to\mathbb{R}$ and $c_U\colon\mathrm{Ham}(M,\omega)\to\mathbb{R}\cup\{-\infty\}$ by
\begin{align*}
&c^U(H)=c(H)+\mathrm{Vol}(U,\omega)^{-1}\cdot\mathrm{Cal}(\phi_H),\\
&c_U(\phi)=\inf\{ c^U(F^1)+\cdots +c^U(F^k) \},
\end{align*}
where the infimum is taken over Hamiltonian functions $F^1,\ldots ,F^k\colon  M\times[0,1]\to\mathbb{R}$ and Hamiltonian diffeomorphisms $h_1,\ldots,h_k$ such that $\mathrm{Supp}(F^i)\subset U$ for any $i$ and  $\phi=h_1\phi_{F^1}h_1^{-1}\cdots h_k \phi_{F^k}h_k^{-1}$.
To prove Proposition \ref{independence}, we use the following lemmas and proposition.
\begin{lem}\label{basic inequality}
For any Hamiltonian function $H\colon M\times[0,1]\to\mathbb{R}$,
\[c^U(H)\leq ||H||.\]
\end{lem}
\begin{proof}
Since $\int_U (H_t-\mathrm{Vol}(U,\omega)^{-1}\cdot\int_UH_t\omega^n)\omega^n=0$ for any $t$,
\[\min_U H_t-\mathrm{Vol}(U,\omega)^{-1}\cdot\int_UH_t\omega^n\leq 0\leq\max_U H_t-\mathrm{Vol}(U,\omega)^{-1}\cdot\int_UH_t\omega^n.\]
By the conditions (3) and (5) of Proposition \ref{independence},
\begin{align*}
c^U(H)
&=c(H)+\mathrm{Vol}(U,\omega)^{-1}\cdot\mathrm{Cal}(\phi_H)\\
& \leq -\int^1_0\min_UH_tdt+c(0)+\mathrm{Vol}(U,\omega)^{-1}\cdot\mathrm{Cal}(\phi_H)\\
&=-\int^1_0(\min_U H_t-\mathrm{Vol}(U,\omega)^{-1}\cdot\int_UH_t\omega^n)dt+c(0)\\
      & \leq -\int^1_0(\min_U H_t-\mathrm{Vol}(U,\omega)^{-1}\cdot\int_UH_t\omega^n)dt\\
      &+\int^1_0(\max_U H_t-\mathrm{Vol}(U,\omega)^{-1}\cdot\int_UH_t\omega^n)dt+c(0)\\
      &=\int_0^1(\max_UH_t-\min_UH_t)dt.
\end{align*}
\end{proof}
\begin{lem}\label{Hofer estimation}
For any Hamiltonian diffeomorphism $\phi$,
\[c_U(\phi)\leq ||\phi||_U.\]
\end{lem}
\begin{proof}
Let $F^1,\ldots,F^k$ and $h_1,\ldots,h_k$ be Hamiltonian functions and Hamiltonian diffeomorphisms such that $\phi=h_1\phi_{F^1}h_1^{-1}\cdots h_k\phi_{F^k}h_k^{-1}$ and $\mathrm{Supp}(F^i)\subset U$ for any $i$, respectively.
By Lemma \ref{basic inequality}, 
\[c^U(F^1)+\cdots +c^U(F^k)\leq||F^1||+\cdots +||F^k||.\]
By taking the infimum, we prove $c_U(\phi)\leq ||\phi||_U$.
\end{proof}
\begin{prop}\label{spec estimation}
For any Hamiltonian function $H\colon M\times[0,1]\to\mathbb{R}$,
\[c(H)+\mathrm{Vol}(M,\omega)^{-1}\mathrm{Cal}(\phi_H)\leq c_U(\phi_H).\]
\end{prop}
\begin{proof}
Let $F^1,\ldots,F^k$ and $h_1,\ldots,h_k$ be Hamiltonian functions and Hamiltonian diffeomorphisms such that $\mathrm{Supp}(F^i)\subset U$ for any $i$ and $\phi_H=h_1\phi_{F^1}h_1^{-1}\cdots h_k\phi_{F^k}h_k^{-1}$, respectively.
Define a Hamiltonian function $F\colon M\times [0,1]\to\mathbb{R}$ by $F(x,t)=F^1(h_1^{-1}x,t)\sharp\cdots \sharp F^k(h_k^{-1}x,t)$.
Then $\phi_F=\phi_H$.
By the conditions (2) and (4) of Proposition \ref{independence},
\begin{align*}
c(F)
& \leq c(F^1\circ h^{-1})+\cdots +c(F^k\circ h^{-1})\\
& = c(F^1)+\cdots +c(F^k).\\
\end{align*}
Since $\phi_F=\phi_H$, by the condition (1) of Proposition \ref{independence} and well-definedness of the Calabi homomorphism, $c(F)=c(H)$ and $\mathrm{Cal}(\phi_F)=\mathrm{Cal}(\phi_H)$.
Thus, since the Calabi homomorphism is a homomorphism,
\begin{align*}
&c(H)+\mathrm{Vol}(U,\omega)^{-1}\mathrm{Cal}(\phi_H)\\
&=c(F)+\mathrm{Vol}(U,\omega)^{-1}\mathrm{Cal}(\phi_F)\\
& \leq c(F^1)+\cdots +c(F^k)+\mathrm{Vol}(U,\omega)^{-1}(\mathrm{Cal}(\phi_{F^1})+\cdots +\mathrm{Cal}(\phi_{F^k}))\\
&=(c(F^1)+\mathrm{Vol}(U,\omega)^{-1}\mathrm{Cal}(\phi_{F^1}))+\cdots +(c(F^k)+\mathrm{Vol}(U,\omega)^{-1}\mathrm{Cal}(\phi_{F^k}))\\
& = c^U(F^1)+\cdots +c^U(F^k).\\
\end{align*}
By taking the infimum, we complete the proof.
\end{proof}
Proposition \ref{independence} immediately  follows from Lemma \ref{Hofer estimation} and Proposition \ref{spec estimation}.

\begin{rem}\label{another c^U}
If we use the convention in Remark \ref{another spec}, we should define $c^U$ by
\[c^U(H)=c(H)-\mathrm{Vol}(U,\omega)^{-1}\cdot\mathrm{Cal}(\phi_H).\]
Then our argument goes well similarly and we prove the inequality in Remark \ref{another spec}.
\end{rem}


\end{document}